\documentclass[11pt,twoside, final]{amsart}
\copyrightinfo{0}{Iranian Mathematical Society}
\usepackage{amsmath,amsthm,amscd,amsfonts,amssymb,enumerate}
\usepackage{graphicx}
\usepackage{color}
\usepackage[colorlinks]{hyperref}

\newtheorem{thm}{Theorem}[section]
\newtheorem{cor}[thm]{Corollary}
\newtheorem{lem}[thm]{Lemma}
\newtheorem{example}[thm]{Example}
\newtheorem{ques}[thm]{Question}
\newtheorem{prop}[thm]{Proposition}
\newtheorem{defn}[thm]{Definition}
\newtheorem{rem}[thm]{\bf{Remark}}

\numberwithin{equation}{section}

\commby{Hamid Pezeshk}
\begin{document}

\title[Amenability of groups and semigroups characterized by configuration]
{Amenability of groups and semigroups characterized by
configuration}

\author[A. Tavakoli]{Ali Tavakoli $^*$}
\address[Ali Tavakoli]{Department of Mathematics, University of Isfahan,  Isfahan, Iran.}
\email{at4300125@gmail.com}

\author[A. Rejali]{Ali Rejali }
\address[Ali Rejali]{Department of Mathematics, University of Isfahan, Isfahan, Iran.}
\email{rejali@sci.ui.ac.ir}

  \thanks{$^*$Corresponding author}
%
\maketitle
%

\begin{abstract}
In 2005,  Abdollahi and Rejali, studied the relations between
paradoxical decompositions and configurations for semigroups. In
the present paper, we introduce another concept of amenability on
semigroups and groups which includes amenability of semigroups and
inner-amenability of groups. We have the previous known results to
semigroups and groups satisfying this concept.
\end{abstract}

\section{\bf Introduction}

The notion of an amenable group was introduced by von Neumann in
1929 in relation with his studies of the Banach-Tarski paradox.
Tarski in 1929  proved the well known alternative theorem: a group
is either amenable or paradoxical. The theory of amenability was
extended in the semigroup setting by Day in the 1950s. Nowadays it
plays a major role not only in Geometric Group Theory, but also in
Functional and Harmonic Analysis, in Ergodic Theory and Dynamical
Systems, and in Operator Algebras. The notion of a configuration
for groups was first introduced by
Rosenblatt and Willis in \cite{rw}, but here, the definition is changed to another form.\\
 Let $G$ be a finitely generated group and $F$ be a non-empty subset of the set $S(G)$ of
all bijective maps on $G$. Let
$\varphi=(\varphi_1,\ldots,\varphi_n)$ be a sequence in $F$ such
that the subgroup $<F>$ generated by $F$ in $S(G)$, is equal to
$<\varphi_1,\ldots,\varphi_n>$ and let
$\mathcal{E}=\{E_1,\ldots,E_m\}$ be a partition of $G$. An
$(n+1)$-tuple $C=(c_0,\ldots,c_n)$, where $c_i\in\{1,\ldots,m\}$
for each $i\in\{0,1,\ldots,n\}$, is called an $F$- configuration
corresponding to the configuration pair $(\varphi,\mathcal{E})$,
if there exist an element $x\in G$ with $x\in E_{c_0}$ such that
$\varphi_i(x)\in E_{c_i}$, for each $i\in\{1,\ldots,n\}$. The set
of all $F$-configurations corresponding to the configuration pair
$(\varphi,\mathcal{E})$ will be denoted by
$Con_F(\varphi,\mathcal{E})$.\\
Let
$x_0(C)=E_{c_0}\cap\varphi_1^{-1}(E_{c_1})\cap\ldots\cap\varphi_n^{-1}(E_{c_n})$
and $x_j(C)=\varphi_j(x_0(C))$, for $C\in
Con_F(\varphi,\mathcal{E})$. Then the $F$-configuration equation
corresponding to the configuration pair $(\varphi,\mathcal{E})$ is
the system of equations
\begin{equation}
\sum \{f_{C}\mid \quad x_{0}(C)\subseteq E_{i}\}=\sum \{f_{C}\mid
\quad x_{j}(C)\subseteq E_{i}\}, \label{clever1}
\end{equation}
where $f_{C}$ is the variable corresponding to the configuration
$C$. This system of equations will be denoted by
$Eq_F(\varphi,\mathcal{E})$. In this case, this equation system is
equivalent to a matrix equation as
\begin{equation}
AX=0, \label{clever2}
\end{equation}
where $A$ is an $nm\times |Con_F(\varphi,\mathcal{E})|$ matrix
whose entries are 0, 1 or -1 and $X$ is the vector $[f_{C}]$,
where $C$ runs over $Con_F(\varphi,\mathcal{E})$.\\
A solution $[f_{C}]$ to $Eq_F(\varphi,\mathcal{E})$ satisfying
$\sum_{C} \{f_{C}\mid C\in Con_F(\varphi,\mathcal E)\}=1$ and
$f_{C}\geq 0$, for all $C\in Con_F(\varphi,\mathcal{E})$ will be
called a normalized solution of the equations system
\eqref{clever1}. The corresponding matrix form whose solution is
normalized, has the form $AX=B$, where $A$ is an $(nm+1)\times
|Con_F(\varphi,\mathcal{E})|$ matrix whose entries are 0, 1 or -1
and all entries of the last row of $A$ are 1. $X$ is the vector
$[f_{C}]$ and $B$ is the vector whose last entry is 1 and all
others are 0. It is well known that, if $A=[a_{i,j}]$, then
$a_{i,j}=1$ [resp. $a_{i,j}=-1$] if and only if $x_i(C)\subseteq
E_j$ and $x_0(C)\nsubseteq E_j$ [resp. $x_i(C)\nsubseteq E_j$ and
$x_0(C)\subseteq E_j$], for some $C\in
Con_F(\varphi,\mathcal{E})$;
otherwise $a_{i,j}=0$.\\
By a non-zero solution of $Eq_F(\varphi,\mathcal{E})$, we mean a
solution $\{f_C\mid C\in Con_F(\varphi,\mathcal{E})\}$ of
$Eq_F(\varphi,\mathcal{E})$ such that $f_C \neq 0$  for some $C\in
Con_F(\varphi,\mathcal{E})$. We show that (see proposition
\ref{prop1}) the  equation in matrix form  has a non-zero solution
if and only if the latter has a normalized solution. It is easy to
see that a matrix equation $AX=0$ has a non-zero solution if and
only if rank($A$) is less than the number of columns of $A$.
Therefore the matrix equation \eqref{clever2} has no non-zero
solution if and only if $rank(A)\leq
|Con_F(\varphi,\mathcal{E})|$.\\
The relation between amenability and configuration of a group was
studied in \cite{rw} and \cite{rty}. Here, we introduce the
concept of $F$-amenability of a group.

\begin{defn}
A group $G$ is called $F$-amenable, if there exist an
$F$-invariant mean $M$ on $\ell_{\infty}(G)$ that is
$M(f\circ\varphi)=M(f)$, for all $f\in \ell_{\infty}(G)$ and
$\varphi\in F$, where $\ell_{\infty}(G)$ denotes the set of all
real valued bounded functions on $G$.
\end{defn}

Now let $G$ be a finitely generated group and $L(G)=\{\lambda_x :
x\in G\}$, where $\lambda_x : G\rightarrow G$ is the left
translation $y\mapsto xy$ for each $y\in G$, and $I(G)=\{I_x :
x\in G\}$ where $I_x : G\rightarrow G$ is the inner automorphism
$y\mapsto x^{-1}yx$. Then, according to our terminology, $G$ is
$L(G)$-amenable [$I(G)$-amenable] if and only if $G$ is amenable
[resp. inner amenable]. In  general, inner amenability is much
weaker than amenability. So, $F$-amenability does not imply
amenability.\\
The configuration which introduced in \cite{rw} can be obtained as
an important special case of our notion. In fact, Rosenblatt and
Willis studied
$$Con(G)=\{Con_F(\varphi,\mathcal{E}) | F \text{ is a finite subset of } L(G) \text{ s.t. } \lambda(G)=<F> \}.$$

\begin{rem} \label{remamena}
Let $F=<\varphi_1,\ldots,\varphi_n>$, for some $\varphi_i\in
S(G)$. Then each $\varphi\in F$ is a finite product of $\varphi_j$
and $\varphi_j^{-1}$.  Let $M$ be a
$\{\varphi_1,\ldots,\varphi_n\}$-invariant mean on
$\ell_{\infty}(G)$. Then
$$M(f)=M((f\circ\varphi_j^{-1})\circ\varphi_j)=M(f\circ\varphi_j^{-1}),$$
for all $f\in \ell_{\infty}(G)$ and $j\in\{1,2,\ldots,n\}$.
Therefore $M$ is an $F$-invariant mean on $\ell_{\infty}(G)$.\\
Now suppose that $F$ is a non-empty subset of $S(G)$, not
necessary
finite. We have the following two facts.\\
(1) if $M$ is an $F$-invariant mean on $\ell_{\infty}(G)$, then
$M$ is an $<F>$-invariant mean on $G$.\\
(2) if $F_1\subseteq F_2$ are non-empty subsets of $S(G)$, then
$F_2$-amenability of $G$ implies  $F_1$-amenability of $G$.\\
\end{rem}

\begin{lem}
Let $F$ be a non-empty subset of $S(G)$, not necessary finite. The
following statements are equivalent.
\begin{enumerate}
\item $G$ is $F$-amenable.\\
\item $G$ is $<\varphi_1,\ldots,\varphi_n>$-amenable, for all finite subset $\{\varphi_1,\ldots,\varphi_n\}$ of $F$.\\
\item $G$ is $\{\varphi_1,\ldots,\varphi_n\}$-amenable, for all
finite subset $\{\varphi_1,\ldots,\varphi_n\}$ of $F$.\\
\end{enumerate}
\end{lem}
\begin{proof}
Due to the remark \ref{remamena}, it is sufficient to prove
(3)$\Rightarrow$(1). \\ Let $\mathrm{T}$ be the family of all
finite non-empty subsets of $F$. Then for every $C\in\mathrm{T}$,
there exists a $C$-invariant mean $M_C$ on $\ell_{\infty}(G)$. If
$\mathrm{T}$ is partially ordered by set inclusion, then, every
$M\in w^*-cl\{M_C\}$ is an $F$-invariant mean on
$\ell_{\infty}(G)$, where $w^*-cl$ means the weakly-$*$ closure.
\end{proof}
In \cite{rw} it is proved that a finitely generated  group $G$ is
amenable if and only if each configuration equation associated to
a configuration pair in $Con(G)$ has a normalized solution. The
link between amenability and normalized solution is seen in
\cite{arw} and certain group properties which can be characterized
by configurations is also studied. In \cite{arw} it is asked
whether the normalized solution can be replaced by a non-zero
solution in the latter. In section 2 we not only give a positive
answer to this question, but also we generalize  it for
$F$-amenability.

\begin{defn} \label{defas}
Let $\{A_1,\ldots,A_n;B_1,\ldots,B_m\}$ be a partition of $G$ such
that there exist two subsets $\{ \varphi_1,\ldots,\varphi_n\}$ and
$\{ \psi_1,\ldots,\psi_m\}$ of $F$ with the following property:
\begin{eqnarray*}
G & = & A_1\cup A_2\cup\ldots\cup A_n\cup B_1\cup B_2\cup\ldots B_m \\
& = & \varphi_1(A_1)\cup\varphi_2(A_2)\cup\ldots\cup\varphi_n(A_n) \\
&  = & \psi_1(B_1)\cup\psi_2(B_2)\cup\ldots\cup\psi_m(B_m).
\end{eqnarray*}
 Then we say that $G$ has an $F$-paradoxical decomposition $(\varphi_i,\psi_j;A_i,B_j)$. In this case, the
$F$-Tarski number of a group $G$ is the minimum of $m+n$, over all
possible $F$-paradoxical decompositions of $G$ and we denote it by
$\tau_F(G)$. If $G$ has no $F$-paradoxical decomposition, we put
$\tau_F(G)=\infty$.
\end{defn}
In section 3, we study the relation between non-$F$-amenability
and having an $F$-paradoxical decomposition for a group.\\
 A dynamical system is a triple $(G,X,\alpha)$, where $\alpha :
G\rightarrow S(X)$ is an action of a group $G$ on a set $X$. The
dynamical system $(G,X,\alpha)$ is amenable if there exists a
finitely additive probability measure $\mu$ defined on the power
set $P(X)$ of the space $X$ which is $\alpha$-invariant, i.e.
$\mu(\alpha_g(A))=\mu(A)$, for all $A\subset X$ and $g\in G$. We
know that the dynamical system $(G,X,\alpha)$ is amenable if and
only if $X$ has no paradoxical decomposition (see \cite{tgc}). Let
$F=\{\alpha_g | g\in G\}$ and $X=G$. Then the dynamical system
$(G,X,\alpha)$ is amenable if and only if $G$ if $F$-amenable.

\section{\bf $F$-Amenability of Groups}

Throughout  this section $G$ is a finitely generated group and $F$
is a non-empty subset of all bijective maps on $G$ such that
$<F>=<\varphi_1,\ldots,\varphi_n>$, where $\varphi_i\in F$, for
$i=1,2,\ldots,n$.
\begin{prop} \label{prop1}
The following statements are equivalent.
\begin{enumerate}
\item $G$ is $F$-amenable.\\
\item Each $F$-configuration equation $Eq_F(\varphi,\mathcal{E})$
has a normalized solution.\\
\item Each $F$-configuration equation $Eq_F(\varphi,\mathcal{E})$
has a non-zero solution.
\end{enumerate}
\end{prop}
\begin{proof}
(1)$\Rightarrow$ (2) Let $M$ be an $F$-invariant mean on
$\ell_{\infty}(G)$. Then $f_C=M(\chi_{x_0(C)})$, for $C\in
Con_F(\varphi,\mathcal{E})$, is a normalized solution of
$Eq_F(\varphi,\mathcal{E})$.\\
(2)$\Rightarrow$(1) Let $(f_C)$ be a normalized solution of
$Eq_F(\varphi,\mathcal{E})$. \\
 Choose $x_C\in x_0(C)$ and define:
\begin{equation*}
f_{(\varphi,\mathcal{E})}(x) = \left\{
\begin{array}{rl}
f_C & \text{if } x = x_C,\\
0 & \text{ otherwise }.\\
\end{array} \right.
\end{equation*}
Then each $M\in w^*-cl\{\hat{f}_{(\varphi,\mathcal{E})}\}$
satisfies $M(f\circ\varphi)=M(f)$, for all $f\in \ell_{\infty}(G)$
and $\varphi\in F$.\\
(3)$\Rightarrow$(2) Let $f\in \ell_1(G)$ be a non-zero solution of
$Eq_F(\varphi,\mathcal{E})$. Define $\Phi\in \ell_{\infty}(G)^*$
by $\Phi(h)=\sum_{x\in G}f(x)h(x)$, for $h\in \ell_{\infty}(G)$.
There exist positive linear functionals $\Phi^+$ and $\Phi^-$ such
that $\Phi=\Phi^+ -\Phi^-$ and $\|\Phi\|=\|\Phi^+\|+\|\Phi^-\|$.
Since $\|\Phi\|=\|f\|_1 \neq 0$, so we can assume $\Phi^+ \neq 0$,
say. By definition,
$$\Phi^+(g)=\sup\{\Phi(h) : 0\leq h\leq g\},$$
for any non-negative function $g$. Furthermore,
\begin{eqnarray*}
\Phi(\chi_{E_i}o \varphi_j)=\Phi(\chi_{\varphi_j^{-1}(E_i)}) & = & \sum_C \{\Phi(\chi_{x_0 (C)}) : x_j(C)\subseteq E_i\} \\
& = & \sum_C \{f_C : x_j(C)\subseteq E_i\} \\
&  = & \sum_C \{f_C : x_0(C)\subseteq E_i\} = \Phi(\chi_{E_i}),
\end{eqnarray*}
for all $i$ and $j$. Thus $\Phi(h\circ\varphi_j)=\Phi(h)$, for all
$h\geq 0$. Therefore:
\begin{eqnarray*}
\Phi^+(\chi_{\varphi_j^{-1}(E_i)}) & = & \sup\{\Phi(h\circ\varphi_j) : 0\leq h\circ\varphi_j\leq\chi_{E_i}\circ\varphi_j\} \\
&  = & \sup\{\Phi(h) : 0\leq h\leq\chi_{E_i}\}=\Phi^+(\chi_{E_i}).
\end{eqnarray*}
Let $k_C=\Phi^+(\chi_{x_0(C)})/\|\Phi^+\|$, then $(k_C)$ is a
normalized solution of $Eq_F(\varphi,\mathcal{E})$.\\
(2)$\Rightarrow$(3) This is trivial.
\end{proof}

\begin{cor}
Let $G_1$ and $G_2$ be finitely generated groups such that
$Con_{F_1}(G_1)=Con_{F_2}(G_2)$. Then $G_1$ is $F_1$-amenable if
and only if $G_2$ is $F_2$-amenable.
\end{cor}

\section{\bf $F$-Paradoxical Decomposition of Groups}

In this section, we generalize Tarski's theorem on amenability for
$F$-amenability of groups. For the special case, set $F=L(G)$.\\
Let $F$ be a subgroup of $S(G)$ under composition operation and
$A,B\subseteq G$. So $A$ and $B$ are $F$-equidecomposable if there
exist partitions $\{A_1,\ldots,A_m\}$ and $ \{B_1,\ldots,B_m\}$ of
$A$ and $B$, respectively, and elements $\varphi_i\in F$ such that
$\varphi_i(A_i)=B_i$ for all $i=1,\ldots,m$. If $A$ and $B$ are
$F$-equidecomposable, then we write $A\cong B$. We say that $A\leq
B$, if $A\cong C$ for some subset $C$ of $B$. It is routine to
show that "$\cong$" is an equivalence relation on power set
$P(G)$. Also a standard Cantor-Bernstein argument shows that $A\leq B$ and $B\leq A$
implies $A\cong B$. \\
Let $S_{\mathbb{N}}$ be the set of all bijective maps on
$\mathbb{N}$. Define $(\varphi,p)(x,n)=(\varphi(x),p(n))$, for
$\varphi\in F$ and $p\in S_{\mathbb{N}}$. Let
$$\mathcal{N}=\{C\subseteq G\times\mathbb{N} : C\subseteq B\times F
\text{ for some } B\subseteq G \text{ and finite set } F\subseteq
\mathbb{N} \}.$$ Then each $N\in \mathcal{N}$ can be written
uniquely in the form $N=\bigcup_{i=1}^n C_i\times \{j_i\}$, where
$1\leq j_1 < j_2 < \ldots <j_n $ and $\emptyset\neq C_i\subseteq
G$.\\
Let $N_1=\bigcup_{i=1}^n C_i\times \{j_i\}$ and
$N_2=\bigcup_{i=1}^n D_i\times \{k_i\}$ be elements of
$\mathcal{N}$. Then $N_1\cong N_2$ if and only if there exist
$\varphi_i\in F$ and $p_i\in S_{\mathbb{N}}$ such that
$\varphi_i(C_i)=D_i$ and $p_i(j_i)=k_i$, for $i\in\{1,2,\ldots
n\}$. Define $\sum=\frac{\mathcal{N}}{\cong}=\{N^{\thicksim} :
N\in \mathcal{N}\}$, where $N^{\thicksim}$ is the equivalence
class of $N$. Choose $h\in F\times S_{\mathbb{N}}$ such that
$h(N_1)\cap N_2=\emptyset$. Then $\sum$ is an abelian semigroup
under addition operation $N_1^{\thicksim} +
N_2^{\thicksim}:=(h(N_1)\cup
N_2)^{\thicksim}$.\\
Define $\alpha=(G\times\{1\})^{\thicksim}$,  so $2\alpha=\alpha +
\alpha = (G\times\{1\}\cup G\times\{2\})^{\thicksim}$.\\
In the following, we show that, $G$ is $F$-amenable if and only if
$\alpha\neq 2\alpha$. A finitely  additive probability measure
$\mu$ of the power set $P(G)$ is called  $F$-invariant, if
$\mu(\phi(A))=\mu(A)$ for all $A\subseteq G$ and $\phi\in F$.
\begin{lem} \label{lem1}
The following statements are equivalent.
\begin{enumerate}
\item $G$ is $F$-amenable.\\
\item There exist an $F\times S_{\mathbb{N}}$-invariant measure
$\mu$ on $\mathcal{N}$ such that $\mu(G\times\{1\})=1$.\\
\item There exist a homomorphism $f : \sum \rightarrow [0,\infty)$
such that $f(\alpha)=1$. \\
\item $\alpha\neq 2\alpha$.
\end{enumerate}
\end{lem}
\begin{proof}
(1)$\Rightarrow$(2) let $\nu$ be an $F$-invariant measure on
$P(G)$. Define $\mu(N)=\sum_{i=1}^n \nu(C_i)$, for each
$N=\cup_{i=1}^n C_i\times \{j_i\}$ in $\mathcal{N}$. Since
$\nu(G)=1$, we have $\mu(G\times\{1\})=1$ and
$$\mu(\varphi\times p(N))=\mu(\bigcup_{i=1}^n \varphi(C_i)\times\{p(j_i)\})
=\sum_{i=1}^n \nu(\varphi(C_i))= \sum_{i=1}^n\nu(C_i)=\mu(N).$$
Hence $\mu$ is an $F\times
S_{\mathbb{N}}$-invariant measure on $\mathcal{N}$.\\
(ii)$\Rightarrow$(i) Let $\mu$ be an $F\times
S_{\mathbb{N}}$-invariant measure on $\mathcal{N}$. Then
$\nu(A)=\mu(A\times\{1\})$ is an $F$-invariant measure on $P(G)$.
Thus $G$ is $F$-amenable.\\
(3)$\Rightarrow$(2) Let $\nu(A)=f(A\times\{1\})^{\thicksim}$, for
$A\subseteq G$. Then
$$\nu(G)=f(G\times\{1\})^{\thicksim}=f(\alpha)=1,$$
and $$\nu(A_1\cup
A_2)=f((A_1\times\{1\})\cup(A_2\times\{1\}))^{\thicksim}=
\nu(A_1)+\nu(A_2),$$
for $A_1 , A_2\subseteq G$ such that $A_1\cap A_2=\emptyset$.\\
(4)$\Rightarrow$(3) Let $T=\{n\alpha : n\in \mathbb{N}\}$ and $F :
T\rightarrow [0,\infty)$ defined by $F(n\alpha)=n$. Then by a
similar argument as is used in \cite{ltp}, p. 119,  $\alpha\neq
2\alpha$ if and only if $k\alpha\neq l\alpha$ whenever $k\neq l$.
 $T$ is a sub-semigroup  of the abelian semigroup $\sum$ and $F(\alpha)=1$;
also $s\leq t$ in $T$ (i.e. $s=t$ or there exist $w\in T$ such
that $s+w=t$) implies $F(s)\leq F(t)$; thus $F$ can be extended to
a homomorphism $f : \sum \rightarrow [0,\infty)$ so
that $f(\alpha)=1$ by \cite{ltp}, p. 117.\\
(1)$\Rightarrow$(3) Let $\nu$ be an $F$-invariant measure in $G$.
Define $f(N^{\thicksim})=\sum_{i=1}^n \nu(C_i)$, for
$N=\cup_{i=1}^n C_i\times \{j_i\}$. Let $N_1=\cup_{i=1}^n
C_i\times \{j_i\}$ and $N_2=\cup_{i=1}^n D_i\times \{k_i\}$ and
$N_1^{\thicksim}=N_2^{\thicksim}$. Then $N_1^{\thicksim}\cong
N_2^{\thicksim}$, so there exist $\varphi_i\in F$ and $p_i\in
S_{\mathbb{N}}$ such that $\varphi_i(C_i)=D_i$ and $p_i(j_i)=k_i$.
Hence $f(N_1^{\thicksim})=\sum_{i=1}^n \nu(C_i)=\sum_{i=1}^n
\nu(\varphi(C_i))=f(N_2^{\thicksim})$, so $f$ is well-defined.\\
Let $h=\varphi\times p\in F\times S_{\mathbb{N}}$, such that
$h(N_1)\cap N_2=\emptyset$. Then:
\begin{eqnarray*}
f(N_1^{\thicksim} + N_2^{\thicksim}) & = & f(\bigcup_{i=1}^n \varphi(C_i)\times\{p(j_i)\}\cup D_i\times\{k_i\}) \\
& = & \sum_{i=1}^n \nu(\varphi(C_i))+\sum_{i=1}^n \nu(D_i)=\sum_{i=1}^n \nu(C_i)+\sum_{i=1}^n \nu(D_i) \\
&  = & f(N_1^{\thicksim})+f( N_2^{\thicksim}).
\end{eqnarray*}
So $f$ is a homomorphism. Clearly,
$f(\alpha)=f((G\times\{1\})^{\thicksim})=\nu(G)=1$.\\
(3)$\Rightarrow$(4) Since $f(\alpha)=1$, so $f(2\alpha)=2$. Thus
$\alpha\neq 2\alpha$. Hence  the proof is complete.
\end{proof}

We now state the main result of this section.

\begin{thm} \label{them1}
The following statements are equivalent.
\begin{enumerate}
\item $G$ is $F$-amenable.\\
\item There exist no $F$-paradoxical decomposition for $G$.
\end{enumerate}
\end{thm}
\begin{proof}
(1)$\Rightarrow$(2) Suppose not! Let $\nu$ be an $F$-invariant
measure for $G$ and $(\varphi_i,\psi_j;A_i,B_j)$ be an
$F$-paradoxical decomposition for $G$. Then:
$$1=\nu(G)=\nu(\bigcup_{i=1}^n \varphi_i (A_i))=\sum_{i=1}^n \nu(\varphi_i(A_i))=\sum_{i=1}^n \nu(A_i).$$
Similarly, $\sum_{j=1}^m \nu(B_j)=1$. Hence,
$$1=\nu(G)=\nu(\bigcup_{i=1}^n A_i)+\nu(\bigcup_{j=1}^m B_j)=1+1=2,$$
which is a contradiction.\\
(2)$\Rightarrow$(1) Suppose not! so by lemma \ref{lem1},
$\alpha=2\alpha$. Then
$G\times\{1\}\cong(G\times\{1\})\cup(G\times\{2\})$. Thus there
exist a partition
$\{A_1\times\{1\},\ldots,A_n\times\{1\};B_1\times\{1\},\ldots,B_m\times\{1\}\}$
of $G\times\{1\}$ and $(\varphi_i,p_i),(\psi_j,q_j)\in F\times
S_{\mathbb{N}}$  such that $p_i(1)=1$ and $q_j(1)=2$ for all $i$
and $j$, so that:
$$G\times\{1\}\bigcup G\times\{2\}=(\bigcup_{i=1}^n \varphi_i\times p_i(A_i\times\{1\}))\bigcup(\bigcup_{j=1}^m
\psi_j\times q_j(B_j\times\{1\}))$$ Thus
$G\times\{1\}=\cup_{i=1}^n \varphi_i(A_i)\times\{1\}$ and
$G\times\{2\}=\cup_{j=1}^m \psi_j(B_j)\times\{2\}$. Hence
$G=\cup\varphi_i(A_i)=\cup\psi_j(B_j)$. So $G$ has an
$F$-paradoxical decomposition, which is a contradiction.
\end{proof}

Similar to \cite{rty}, we are interested to construct an
$F$-paradoxical decomposition for non-$F$-amenable groups by using
$F$-configuration equations
and conversely.\\
Let $(\varphi_i,\psi_j;A_i,B_j)$ be an $F$-paradoxical
decomposition of $G$ and $f\in \ell_1^+(G)$. Then:
\begin{eqnarray*}
\|f\|_1 & = & \sum_C\{f_C : C\in Con_F(\varphi,\mathcal{E})\} \\
& = & \sum_C\sum_{i=1}^n \{f_C : x_0(C)\subseteq A_i\}+\sum_C\sum_{j=1}^m \{f_C : x_0(C)\subseteq B_j\} \\
&  = & \sum_C\sum_{i=1}^n \{f_C : x_i(C)\subseteq A_i\}+\sum_C\sum_{j=1}^m \{f_C : x_j(C)\subseteq B_j\} \\
& = & 2\sum_C\{f_C : C\in Con_F(\varphi,\mathcal{E})\}=2\|f\|_1,
\end{eqnarray*}
where $\varphi=(\varphi_1,\ldots,\varphi_n;\psi_1,\ldots,\psi_m)$
and $\mathcal{E}=\{A_1,\ldots,A_n;B_1,\ldots,B_m\}$. Therefore
$Eq_F(\varphi,\mathcal{E})$ has no non-zero solutions.\\
Suppose $Eq_F(\varphi,\mathcal{E})$ has no non-zero solutions.
Suppose $Con_F(\varphi,\mathcal{E})=\{D_1,\ldots,D_s\}$ such that
$E_1=\cup_{i=1}^{r_1}x_0(D_i),E_2=\cup_{i=r_1+1}^{r_1+r_2}x_0(D_i)$
and so on. Define $\mathcal{E}^{\prime}=\{E^{\prime}_i :
i=1,\ldots,s\}$ where $E^{\prime}_i=x_0(D_i)$ for each $i$. Then
$Eq_F(\varphi,\mathcal{E}^{\prime})$ has a non-zero solutions.
Similarly, if $Eq_F(\varphi,\mathcal{E}^{\prime})$ has a non-zero
solutions then $Eq_F(\varphi,\mathcal{E})$ has a non-zero
solutions.

\begin{ques}
Let $Eq_F(\varphi,\mathcal{E})$ be a system of equations having no
non-zero solution for some configuration pair
$(\varphi,\mathcal{E})$. How can "explicitly"  construct an
$F$-paradoxical decomposition from $Eq_F(\varphi,\mathcal{E})$?
\end{ques}

It is to be noted that $G=<g_1,g_2,\ldots,g_l>$ is non-amenable if
and only if the equation $|g_i^{-1}E_j\cap X|=|E_i\cap X|$, for
$1\leq i\leq l , 1\leq j\leq m$, has no non-empty finite solution
$X$ in $G$, for some partition $\mathcal{E}=\{E_1,\ldots,E_m\}$.
\begin{example}
\cite{rw}. Let $G=<g_1,g_2>$ be the free group on two (free)
generators $g_1,g_2$ and $E_i$ be the set of all reduced words
starting $g_i$, for $i=1,2$, and $E_3=G-(E_1\cup E_2)$. Then
$Eq(\varphi,\mathcal{E})$ has no non-zero solution. In compare to
the above notations, let
$$(\varphi_i)=(1,\lambda_{g_1},\lambda_{g_1}),
(\psi_j)=(1,1,1,\lambda_{g_2},\lambda_{g_2})$$ and
$$(A_i)=(E^{\prime}_1,E^{\prime}_2,E^{\prime}_5),
(B_j)=(E^{\prime}_2,E^{\prime}_3,A,E^{\prime}_7,B),$$ for some
$A\subseteq E^{\prime}_6$ and $B=E^{\prime}_6-A$. Then
$(\varphi_i,\psi_j;A_i,B_j)$ is a paradoxical decomposition of
$G$.
\end{example}

\section{\bf $F$-Amenability of Semigroups}

In this section,  a new type of amenability for semigroups is
introduced. Also the notion of an $F$-paradoxical decomposition
for semigroups which was asked by Paterson in special case in
\cite{ltp} p. 120, is defined. We find the relation between the
existence of $F$-paradoxical decompositions and
non-$F$-amenability
for semigroups. The definition is almost similar to that of groups, we bring it for completeness.\\

Let $S$ be a discrete semigroup and $A\subseteq S$. For any map
$f:S\rightarrow S$ (not necessary invertible), recall that
$f^{-1}(A)=\{t\in S : f(t)\in A \}$. Let
$\varphi=(\varphi_1,\ldots,\varphi_n)$ be an $n$-tuple of the
functions (not necessary invertible) on $S$ and
$\mathcal{E}_0=\{E_1,\ldots,E_m\}$ be a partition of $S$. An
$(n+1)$-tuple $C=(c_0,\ldots,c_n)$, where $c_i\in\{1,\ldots,m\}$
for each $i\in\{0,1,\ldots,n\}$, is called a configuration
corresponding to the configuration pair $(\varphi,\mathcal{E}_0)$,
if there exist an element $x\in S$ with $x\in E_{c_0}$ such that
$\varphi_i(x)\in E_{c_i}$, for each $i\in\{1,\ldots,n\}$. The set
of all configurations corresponding to the configuration pair
$(\varphi,\mathcal{E}_0)$ will be denoted by
$Con(\varphi,\mathcal{E}_0)$. Let
$\mathcal{E}_i=\{\varphi_i^{-1}(E_j): j\in\{1,\ldots,m\}\}$, for
each $i\in\{1,\ldots,n\}$. Then $\mathcal{E}_i$ is a partition of
$S$ for each $i=1,2,\ldots,n$. (We remove empty elements from these collections.)\\
 Let $x_0(C)=E_{c_0}\cap\varphi_1^{-1}(E_{c_1})\cap\ldots\cap\varphi_n^{-1}(E_{c_n})$
and $x_j(C)=\varphi_j(x_0(C))$. \\
 Let $F$ be a non-empty subset of the set of all
maps $S^S$ on $S$. For $n$-tuples
$\varphi=(\varphi_1,\ldots,\varphi_n)$ in $F$ such that the
semigroup $<F>$ generated by $F$ in $S^S$ is equal to
$<\varphi_1,\ldots,\varphi_n>$, we call a configuration
corresponding to the configuration pair $(\varphi,\mathcal{E})$ is
denoted by $Con_F (\varphi,\mathcal{E})$. Then the
$F$-configuration equations corresponding to the configuration
pair $(\varphi,\mathcal{E})$ are defined similarly to the previous
case. (equations \eqref{clever1}) \\
A semigroup $S$ is called $F$-amenable, if there exists an
$F$-invariant mean $M$ on $\ell_{\infty}(S)$, that is
$M(f\circ\varphi)=M(f)$, for all $f\in \ell_{\infty}(S)$ and
$\varphi\in F$, where $\ell_{\infty}(S)$ denoted the set of all
real valued bounded functions on $S$.\\

 Adler and  Hamilton, \cite{ah}, showed that $S$ is left
amenable if and only if $S$ satisfies the following left invariant
condition:\\
for any sequence $(s_1,\ldots,s_n)$ in $S$ and for all sequences
$(A_1,\ldots,A_n)$ of subsets in $S$ there exists a non-empty
finite set $X\subseteq S$ such that $|s_i^{-1}A_i \cap X|=|A_i\cap
X|$ for all $i\in\{1,2,\ldots,n\}$.\\ We prove that $S$ is
$F$-amenable if and only if $S$ satisfies the $F$-invariant
condition. \\

\begin{defn}
 Let $\{A_1,\ldots,A_n;B_1,\ldots,B_m\}$ be a partition of
 semigroup $S$ and there exist two subsets $\{ \varphi_1,\ldots,\varphi_n\}$ and $\{
\psi_1,\ldots,\psi_m\}$ of $F$ such that the sets
 $\{\varphi^{-1}_1(A_1),\ldots,\varphi^{-1}_n(A_n)\}$ and
 $\{\psi^{-1}_1(B_1),\ldots,\psi^{-1}_m(B_m)\}$
 are two partitions of $S$. Then we say that $S$
 admits an $F$-paradoxical decomposition
 $(\varphi_i,\psi_j;A_i,B_j)$.
In this case, the $F$-Tarski number of a semigroup $S$ is the
minimum of $m+n$, over all possible $F$-paradoxical decompositions
of $S$.
\end{defn}

 We show that the $F$-Tarski number for semigroups can be 2; however the
corresponding number for groups is at least 4. At first, by a
similar argument as in used in proposition \ref{prop1}, the
following proposition is immediate.
\begin{prop}
The following statements are equivalent.
\begin{enumerate}
\item $S$ is $F$-amenable.\\
\item Each $F$-configuration equation $Eq_F(\varphi,\mathcal{E})$
has a normalized solution.\\
\item Each $F$-configuration equation $Eq_F(\varphi,\mathcal{E})$
has a non-zero solution.
\end{enumerate}
\end{prop}

\begin{lem} \label{lem2}
The following statements are equivalent.
\begin{enumerate}
\item $S$ is $F$-amenable.\\
\item For any sequence $(\varphi_1,\ldots,\varphi_k)$ in $F$ and
for all sequence $(A_1,\ldots,A_k)$ of subsets in $S$, there exist
a finite non-empty subset $X\subseteq S$ such that,
$$|\varphi_i^{-1}(A_i)\cap X| = |A_i\cap X|, \text{ for all }  i=1,\ldots,k.$$
\item For any sequence $(\varphi_1,\ldots,\varphi_n)$ in $F$ and
for each partition $\{E_1,\ldots,E_m\}$ of $S$, there exist a
non-empty finite subset $X\subseteq S$ such that,
$$|\varphi_i^{-1}(E_j)\cap X| = |E_j\cap X|, \text{ for all }  i,j .$$
\end{enumerate}
\end{lem}
\begin{proof}
(2)$\Rightarrow$(3) Let $(\varphi_1,\ldots,\varphi_n)$ be a
sequence in $F$ and $\{E_1,\ldots,E_m\}$ be a partition of $S$.
Put
$$A_j=E_j, A_{m+j}=E_j, \ldots, A_{(n-1)m+j}=E_j \text{ for
all } j=1,\ldots,m.$$ Put also,
$$\varphi^{\prime}_j=\varphi_1, \varphi^{\prime}_{m+j}=\varphi_2, \ldots, \varphi^{\prime}_{(n-1)m+j}=\varphi_n
\text{ for all } j=1,\ldots,m.$$ Then for
$(\varphi^{\prime}_1,\ldots,\varphi^{\prime}_{mn})$ and
$(A_1,\ldots,A_{mn})$, there exists a non-empty subset $X\subseteq
S$ such that,
$$|\varphi_i^{-1}(E_j)\cap X| = |E_j\cap X|,$$
for $i\in\{1,2,\ldots,n\}$ and $j\in \{1,2,\ldots,m\}$.\\
(3)$\Rightarrow$(2) Let $(\varphi_1,\ldots,\varphi_k)$ be a
sequence in $F$ and $(A_1,\ldots,A_k)$ be a sequence of subsets in
$S$. Let $\mathcal{E}_i=\{A_i,A^c_i\}$, for $i=1,\ldots,k$ and
$\mathcal{E}$ be the family of all $n$-tuple intersections on
$\mathcal{E}_i$. Clearly, the cardinality of $\mathcal{E}$ is
$2^n$ and it is a partition of $S$. By (3), There exist a finite,
non-empty subset $X\subseteq S$ so that $|\varphi_i^{-1}(E)\cap X|
= |E\cap X|$ for all $i=1,\ldots,k$ and $E\in \mathcal{E}$. Then
one can show that easily $|\varphi_i^{-1}(A_i)\cap X| = |A_i\cap
X|$, for all $i=1,\ldots,k$.\\ For example, if $k=1$, then
$\mathcal{E}=\{A_1,A^c_1\}$ and there exist a finite, non-empty
subset $X\subseteq S$ such that $|\varphi_1^{-1}(A_1)\cap X| =
|A_1\cap X|$. Also, if $k=2$, then $\mathcal{E}=\{A_1\cap
A_2,A_1\cap A^c_2,A^c_1\cap A_2,A^c_1\cap A^c_2\}$. Hence, there
exist a finite non-empty subset $X\subseteq S$ such that,
$$|\varphi_i^{-1}(E)\cap X| = |E\cap X|,$$
for all $i\in\{1,2\}$ and $E\in \mathcal{E}$. Now we have:
\begin{eqnarray*}
|\varphi_1^{-1}(A_1)\cap X| & = & |\varphi_1^{-1}(A_1\cap A_2)\cap X|+|\varphi_1^{-1}(A_1\cap A^c_2)\cap X| \\
& = & |(A_1\cap A_2)\cap X|+|(A_1\cap A^c_2)\cap X| = |A_1\cap X|.
\end{eqnarray*}
Similarly, $|\varphi_2^{-1}(A_2)\cap X|=|A_2\cap X|$. \\
This completes the proof of (2).\\
(3)$\Rightarrow$(1) Suppose $\mathcal{E}=\{E_1,\ldots,E_m\}$ is a
partition of $S$ and $\varphi=(\varphi_1,\ldots,\varphi_n)$ is a
sequence in $F$.
 Then there exist a non-empty finite subset
$X\subseteq S$ such that,
$$|\varphi_i^{-1}(E_j)\cap X| = |E_j\cap X|, \text{ for all }  i,j .$$
Let
$$f_C=\frac{1}{|X|}|X\cap x_0(C)|, \text{ for all } C\in Con_F(\varphi,\mathcal{E}).$$
Therefore, $[f_C]$ is a normalized solution. In fact:
\begin{eqnarray*}
\sum\{f_C : x_i(C)\subseteq E_j\} & = &
\frac{1}{|X|}|\varphi^{-1}_i(E_j)\cap X|\\
 & = & \frac{1}{|X|}|E_j\cap X|=\sum\{f_C : x_0(C)\subseteq E_j\}.
\end{eqnarray*}
Hence, $S$ is $F$-amenable.\\
(1)$\Rightarrow$(2) See \cite{ah}.
\end{proof}
 The condition (ii) of lemma \ref{lem2}, is called $F$-invariant
condition of semigroup $S$. In the following, we extend
$F$-paradoxical decomposition for semigroups for which was asked
in \cite{ltp} p. 120.\\
 Now, suppose that the identity function $I:S\rightarrow S$ belongs to $F$ and $A,B\subseteq S$;
 then $A$ and $B$ are $F$-equidecomposable and write $A\cong B$, if there exist partitions
 $\{A_1,\ldots,A_n\}$ of $A$ and $\{B_1,\ldots,B_n\}$ of $B$, and
 elements $\varphi_i , \psi_i$ in $F$ such that
 $\varphi^{-1}_i(A_i)=B_i$ and  $\psi^{-1}_i(B_i)=A_i$ for all $i\in
 \{1,2,\ldots,n\}$. It is clear that the relation "$\cong$" is an equivalence relation on power set
 $P(S)$. \\
 We say also that a finitely  additive probability measure $\mu$ of the power set
 $P(S)$ is an $F$-invariant measure if $\mu(\varphi^{-1}(E))=\mu(E)$ for all $\varphi\in F$
and $E\subseteq S$. By an argument as in lemma \ref{lem1}, one can
show that $S$ is $F$-amenable if and only if $\alpha\neq 2\alpha$,
where $\alpha=(S\times\{1\})^{\thicksim}$.
\begin{lem}
The following statements are equivalent.
\begin{enumerate}
\item $S$ is not $F$-amenable.\\
\item $S$ admits an $F$-paradoxical decomposition.
\end{enumerate}
\end{lem}
\begin{proof}
(2)$\Rightarrow$(1) Let $(\varphi_i,\psi_j;A_i,B_j)$ be an
$F$-paradoxical decomposition of $S$ and suppose by contradiction
that $M$ is an $F$-invariant mean on $\ell_{\infty}(S)$. Then
$$1=M(1)=\sum_{i=1}^n M(\chi_{A_i}\circ\varphi_i)=\sum_{i=1}^n M(\chi_{A_i}).$$
Similarly, $\sum_{j=1}^m M(\chi_{B_j})=1$. Since
$\{A_1,\ldots,A_n;B_1,\ldots,B_m\}$ is a partition of $S$, we
deduce that $1=\sum_{i=1}^n M(\chi_{A_i})+\sum_{j=1}^m
M(\chi_{B_j})=2$, which gives a contradiction.\\
(1)$\Rightarrow$(2) It is by a similar argument as is used in
theorem \ref{them1}.
\end{proof}

\begin{rem}
Let $(s_i,t_j;A_i,B_j)$ be an $F$-paradoxical decomposition of
semigroup $S$ so that $|s^{-1}_i A_i\cap X|=|A_i\cap X|$ and
$|t^{-1}_j B_j\cap X|=|B_j\cap X|$, for all $i,j$, for some
non-empty subset $X\subseteq S$. Then:
 $$|X|=\sum_i|A_i\cap X|+\sum_j|B_j\cap X|=\sum_i|s^{-1}_i A_i\cap X|+\sum_j|t^{-1}_j B_j\cap X|=2|X|,$$
 hence, $X$ is empty.
 \end{rem}

Since the existence of $F$-invariant mean is independent of
generating sequence of $F$, the following statement is immediate.
\begin{cor}
Let $F=<\varphi_1,\ldots,\varphi_n>$; the following statements are
equivalent.
\begin{enumerate}
\item $S$ is $F$-amenable.\\
\item For any partition
$\{E_1,\ldots,E_m\}$, there exist a non-empty finite subset
$X\subseteq S$ such that,
$$|\varphi_i^{-1}(E_j)\cap X| = |E_j\cap X|, \text{ for all }  i,j.$$
\item For any partition $\mathcal{E}=\{E_1,\ldots,E_m\}$ of $S$,
there exist a non-empty finite subset $X\subseteq S$ such that,
$$|\varphi_i^{-1}(x_0(C))\cap X| = |x_0(C)\cap X|, $$
for all $i\in\{1,\ldots,n\}$ and $C\in
Con_F(\varphi,\mathcal{E})$, where
$\varphi=(\varphi_1,\ldots,\varphi_n)$.
\end{enumerate}
\end{cor}

\begin{example}
(1) Let $S=(\mathbb{N},\cdot)$ and $x\cdot y=x$ for $x,y\in S$.
Then ${}_x f=f(x)1$ for $f\in \ell_{\infty}(S)$. So $S$ is not
left-amenable and $S=E_1\cup E_2=g^{-1}_1 E_1=g^{-1}_2 E_2$, where
$E_1=2\mathbb{N}$, $E_2=2\mathbb{N}+1$, $g_1=2$ and $g_2=3$. Hence
$S$ has a paradoxical decomposition of Tarski number 2, see \cite{tgc}.\\
(2) Let $S=(\mathbb{N},\circ)$ and $x\circ y=y$, for $x,y\in S$.
Then ${}_x f=f$, for $f\in \ell_{\infty}(S)$. Then $S$ is
left-amenable and  $g^{-1}E=E$, for all $g\in S$ and $E\subseteq
S$. Hence $S$ has no paradoxical decompositions.
\end{example}

\begin{center}
{\textbf{Acknowledgments}}
\end{center}
The authors would like to thank the referee of the paper for his
very careful reading and his invaluable comments. They also thank
the Banach Algebra Center of Excellence for Mathematics at the University of Isfahan. \\


\end{document}